\documentclass[12pt]{article}  
\pdfoutput=1 	
\usepackage{geometry}                		
\usepackage[parfill]{parskip}    		
\usepackage{graphicx}				
\usepackage[hyperindex=true,pagebackref,colorlinks=true,pdfpagemode=none,urlcolor=blue,linkcolor=blue,citecolor=blue,pdfstartview=FitH]{hyperref}
\usepackage{amsmath,amsfonts}
\usepackage{color}
\usepackage{amsthm}
\usepackage{boolexpr}
\usepackage{amssymb,latexsym,pstricks,pst-plot}
\usepackage[english]{babel}
\usepackage{pgf}
\usepackage{tikz}
\usetikzlibrary{arrows,automata}
\usepackage[latin1]{inputenc}
\usepackage[UKenglish]{isodate}
\usepackage[T1]{fontenc}
\usepackage{listings}
\usepackage{xcolor}
\allowdisplaybreaks

\usepackage{amssymb}
\newtheorem{theorem}{Theorem}[section]
\newtheorem{lemma}[theorem]{Lemma}

\newtheorem{corollary}[theorem]{Corollary}

\theoremstyle{definition}

\newtheorem*{remark}{Remarks}

\title{Factorials and Legendre's three-square theorem: II}
\author{Rob Burns}

\begin{document}
\maketitle
\begin{abstract}
Let $\bar{S}$ denote the set of integers $n$ such that $n!$ cannot be written as a sum of three squares. Let $\bar{S}(n)$ denote $\bar{S} \cap [1, n]$. We establish an exact formula for $\bar{S}(2^k)$ and show that $\bar{S}(n) = 1/8*n +  \mathcal{O}(\sqrt{n})$. We also list the lengths of gaps appearing in $\bar{S}$. We make use of the software package Walnut to establish these results.
\end{abstract}

\section{Introduction}
\label{intro}

In a previous paper \cite{https://doi.org/10.48550/arxiv.2101.01567}, we constructed an automaton accepting integers $n$ such that $n!$ cannot be written as a sum of three squares. In the current paper we use this automaton and the software package Walnut to explore the sets
$$
S := \{ n: n!  \, \, \text{is a sum of three squares} \},
$$
$$
\bar{S} := \{ n: n!  \, \, \text{is not a sum of three squares} \},
$$
$$
S(n) := \{ x \leq n : x!  \, \, \text{is a sum of three squares} \}.
$$
and
$$
\bar{S}(n) := \{ x \leq n : x!  \, \, \text{is not a sum of three squares} \}.
$$

In particular, we establish a formula for $S(2^k)$, or equivalently, $\bar{S}(2^k)$, and show that $\bar{S}(n) = 1/8*n +  \mathcal{O}(\sqrt{n})$. We also list the lengths of possible gaps in $S$ and its complement $\bar{S}$.

We mention here some similar previous results. In \cite{Deshouillers_2010}, Deshouillers and Luca established the two equivalent bounds
$$
S(n) = 7 n / 8 +  \mathcal{O}(n^{2/3}) \, \, \text{ and } \,\, \bar{S}(n) = x / 8 +  \mathcal{O}(n^{2/3}).
$$
This bound was improved by Hajdu and Papp \cite{Hajda2018} to 
$$
\bar{S}(n) = x / 8 +  \mathcal{O}(n^{1/2} \log^2 n).
$$

Hajdu and Papp \cite{Hajda2018} proved that the largest gap between consecutive elements of $\bar{S}$ is $42$. In theorem \ref{gaptheorem}, we list the lengths of all gaps that occur between consecutive elements of $S$ and $\bar{S}$. 

We make use of the software package Walnut. Hamoon Mousavi, who wrote the program, has provided an introductory article \cite{https://doi.org/10.48550/arxiv.1603.06017}. Papers that have used Walnut include \cite{Mousavi2016DecisionAF}, \cite{https://doi.org/10.48550/arxiv.2103.10904}, \cite{https://doi.org/10.48550/arxiv.2110.06244}, \cite{Go__2013}, \cite{DU2017146}. Further resources related to Walnut can be found at Jeffrey Shallit's page
\begin{verbatim}
https://www.cs.uwaterloo.ca/~shallit/papers.html   .
\end{verbatim}

The free open-source mathematics software system SageMath \cite{sagemath} was essential for the matrix algebra and factoring of polynomials.

In the next section, we describe the automaton that will be used to examine $\bar{S}$ and $\bar{S}(x)$.

\bigskip

\section{The automaton}
\label{automaton}

We first summarise the conditions satisfied by integers in $\bar{S}$. These come from \cite{https://doi.org/10.48550/arxiv.2101.01567} but were previously established by Granville and Zhu in 1990 \cite{10.2307/2323831}.  Let $n \in \mathbb{N}$, with binary representation given by $n = \sum_{k \geq 0} a_k 2^k$, where all but finitely many $a_i$ are zero. Let $\gamma$ be the highest power of $2$ dividing $n!$. For $i \in \{3,5 \}$, define $\alpha_i = \alpha_i(n)$ by

\begin{align}
\label{alpha3}
\alpha_3 = \alpha_3(n) &:= \#\Big\{ k \geq 0: \sum_{i=k}^{k+2} a_i 2^{i-k} \in \{3,4 \}  \Big\} \\
\label{alpha5}
\alpha_5 = \alpha_5(n) &:= \#\Big\{ k \geq 0: \sum_{i=k}^{k+2} a_i 2^{i-k} \in \{5,6 \} \Big\} .
\end{align}

\bigskip

\begin{theorem}
\label{mainthm}
Let $n \in \mathbb{N}$ have binary representation $n = \sum_{k \geq 0} a_k 2^k$, where all but finitely many $a_i$ are zero. If $\gamma$ is the highest power of $2$ dividing $n!$, then $n! = 2^{\gamma}Z$, where $Z$ satisfies 
\begin{align*}
Z \equiv 3^{\alpha_3(n)}(-1)^{\alpha_5(n)} \pmod{8}
\end{align*}
\end{theorem}

\bigskip

\begin{corollary}
If $n \in \mathbb{N}$, then $n!$ cannot be written as a sum of three squares if and only if $\gamma$ and $\alpha_3$ are even and $\alpha_5$ is odd.
\end{corollary}

\bigskip

According to the above two results, whether $n!$ can be written as a sum of three squares is determined by the triple $( \gamma \pmod{2}, \alpha_3 \pmod{2}, \alpha_5 \pmod{2})$. For convenience, we will give this triple a name. For integers $n$, define $\Theta (n)$ by
$$
\Theta (n) :=  ( \, \, \gamma \pmod{2}, \, \, \alpha_3 \pmod{2}, \, \, \alpha_5 \pmod{2} \, \,).
$$
The components of $\Theta (n)$ are 
\begin{equation}
\label{thetacomp}
\Theta_1(n) :=  \gamma \pmod{2}, \,\,\, \Theta_2 (n) := \alpha_3 \pmod{2}, \,\,\, \Theta_3 (n) := \alpha_5 \pmod{2} .
\end{equation}
Then, $n \in \bar{S}$ if and only if $\Theta (n) = (0, 0, 1)$.

We present here three automata which take as input the binary digits of an integer $n$, starting with the least significant digit. The automata were produced by Walnut based on the instructions provided in \cite{https://doi.org/10.48550/arxiv.2101.01567}. The first automaton, which we call gamma, is pictured in figure \ref{autog}. It accepts integers $n$ if and only if $\gamma(n)$ is even. The second automaton, which we call a3, is pictured in figure \ref{autoa3}. It accepts integers $n$ if and only if $\alpha_3(n)$ is even. The third automaton, which we call a5, is pictured in figure \ref{autoa5}. It accepts integers $n$ if and only if the $\alpha5(n)$ is even. 

\begin{figure}
\centering
\includegraphics[width=9cm]{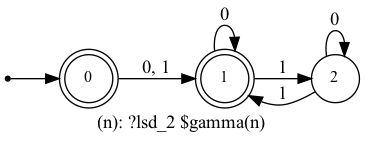}
\caption{Automaton for the parity of $\gamma$}
\label{autog}
\end{figure}

\begin{figure}
\centering
\includegraphics[width=11cm]{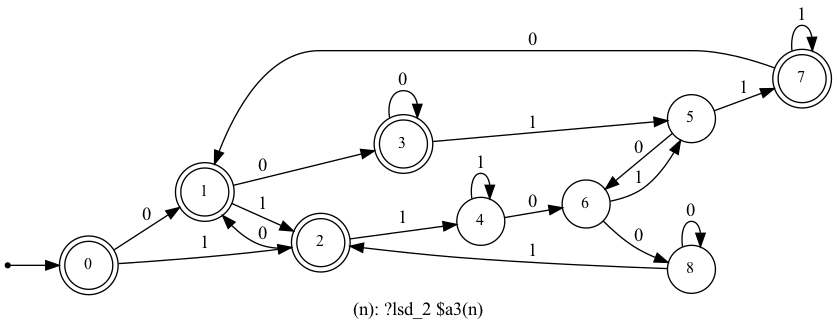}
\caption{Automaton for the parity of $\alpha_3$}
\label{autoa3}
\end{figure}

\begin{figure}
\centering
\includegraphics[width=11cm]{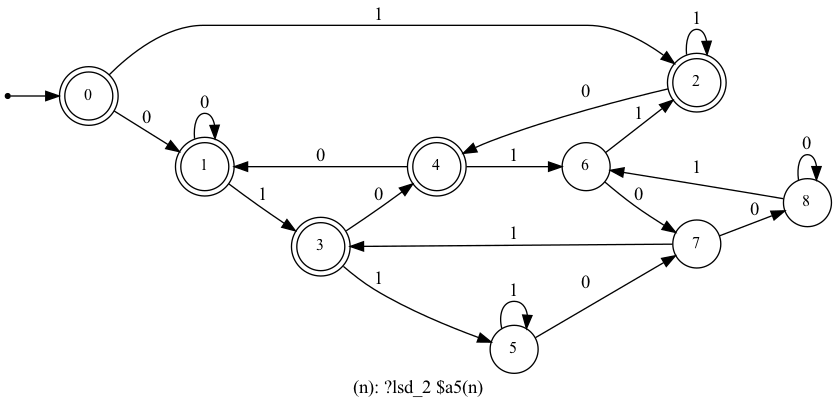}
\caption{Automaton for the parity of $\alpha_5$}
\label{autoa5}
\end{figure}

For any triple of binary digits $(x, y, z)$, we can use Walnut to create an automaton which accepts $n$ if and only if $\Theta(n) = (x, y, z)$. We are particularly interested in determining when $\Theta(n) = (0, 0, 1)$. We create this automaton using the Walnut command:

\begin{verbatim}
def factauto "?lsd_2 $gamma(n) & $a3(n) & ~$a5(n)":
\end{verbatim}

The resulting automaton has 33 states and is pictured in figure \ref{factauto}. 

\begin{figure}
\centering
\includegraphics[width=13cm]{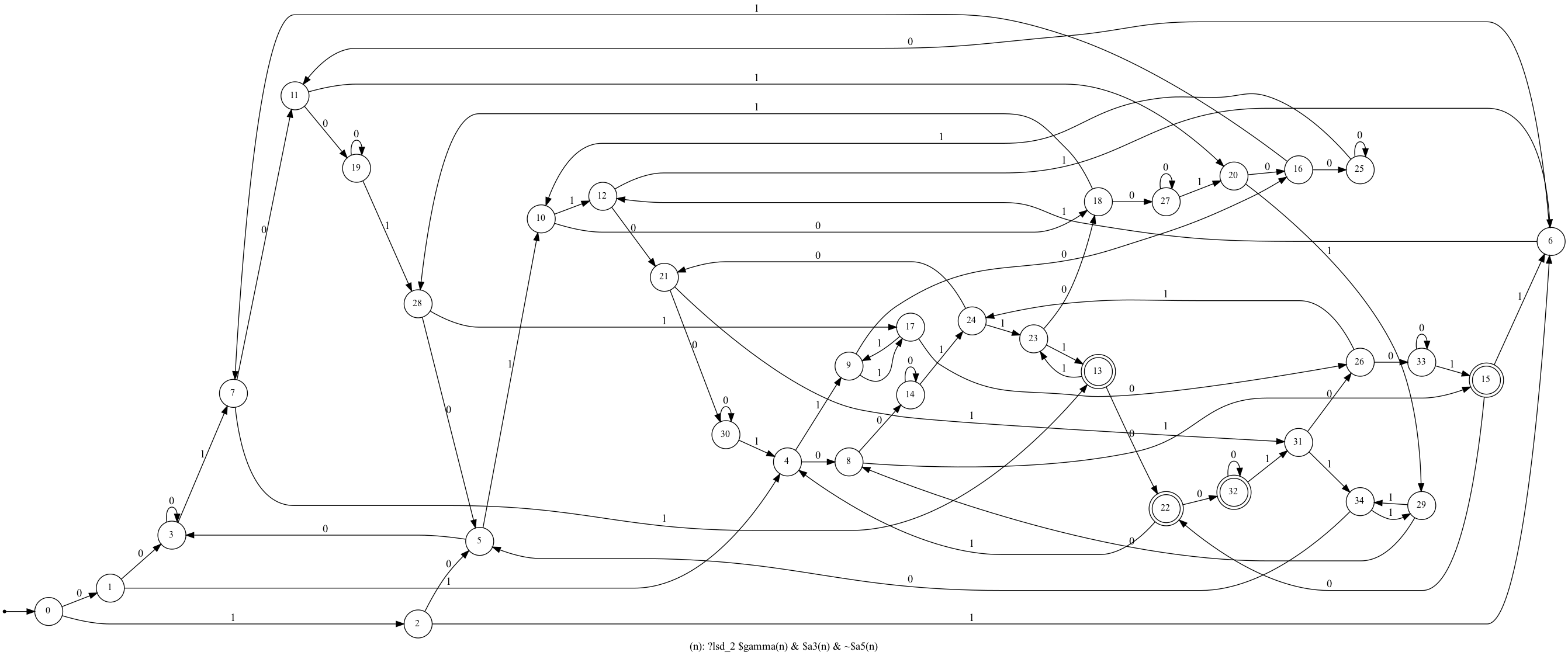}
\caption{Automaton accepting $n$ if $n!$ cannot be written as a sum of $3$ squares}
\label{factauto}
\end{figure}

\bigskip

\section{Gaps and runs in $S$ and $\bar{S}$}
\label{gaps}

In this section we investigate the gaps and runs in the sets $S$ and $\bar{S}$. A gap $g$ exists in $S$ (or $\bar{S}$) if there are integers $x$ and $y$ which are consecutive elements of $S$ and such that $y - x = g$. A run of length $r$ exists in $S$ (or $\bar{S}$) if $S$ contains $x, x + 1, \dots , x + r -1$ but $x - 1$ and $x+r$ are not in $S$. Since $S \cup \bar{S} =  \mathbb{N}$, gaps in $S$ correspond to runs in $\bar{S}$ and vice versa. For example, if $S$ contains a run of size $r$, then $\bar{S}$ contains a gap of length $r+1$.

In  \cite{Hajda2018}, Hajdu and Papp proved that the largest gap between consecutive elements of $\bar{S}$ is $42$. 

\bigskip

\begin{theorem}
\label{gaptheorem}
The gap sizes in $S$ are $\{ 1, 2, 3, 4 \}$. The gap sizes in $\bar{S}$ are
$$
\{ 1, 2, 3, \dots , 23, 25, 26, 28, 30, 31, 33, 34, 35, 37, 38, 42 \}.
$$
\end{theorem}
\begin{proof}
We firstly deal with $\bar{S}$. Recall the automaton "factauto" which was defined previously. It accepts integers $n$ such that $n!$ cannot be written as a sum of three squares. We use Walnut to define an automaton, which we call "gaps":

\begin{lstlisting}[breaklines]
def gaps "?lsd_2 $factauto(n) & $factauto(n+r) &  (Aj  (j < r-1) => ~$factauto(n+ j + 1))":
\end{lstlisting}

This automaton takes as input the integer pair $(n, r)$ and reaches an accepting state if and only if $n \in \bar{S}$ and the next integer in $\bar{S}$ is $n+r$, i.e. there a gap of length $r$ in $\bar{S}$ starting at $n$. Since the automaton has 319 states, we will not display it here. The following Walnut command creates an automaton which accepts an integer $r$ if and only if there is a gap of length $r$ somewhere in $\bar{S}$:
\begin{verbatim}
eval tmp "?lsd_2 E n $gaps(n, r)":
\end{verbatim}
This automaton is displayed in figure \ref{sbargaps}. A patient reader will see that the automaton accepts the integers listed in the theorem.

We can use the same approach to find all gaps in the set $S$. We define the automaton "sgaps" which takes as input the integer pair $(n, r)$ and reaches an accepting state if and only if $n \in S$ and the next integer in $S$ is $n+r$, i.e. there a gap of length $r$ in $S$ starting at $n$. 
\begin{lstlisting}[breaklines]
def sgaps "?lsd_2 ~$factauto(n) & ~$factauto(n+r) &  (Aj  (j < r-1) => $factauto(n+ j + 1))":
\end{lstlisting}
The resulting automaton has 203 states. The following Walnut command creates an automaton which accepts an integer $r$ if and only if there is a gap of length $r$ somewhere in $S$:
\begin{verbatim}
eval tmp "?lsd_2 E n $sgaps(n, r)":
\end{verbatim}
This automaton is displayed in figure \ref{sgaps}. It is clear that it only accepts the integers $1, 2, 3, 4$.
\end{proof}

\bigskip

\begin{remark}
A gap of 42 in $\bar{S}$ commences at $n = 23268$ as mentioned by Hajdu and Papp. This is the second last gap to appear in the sequence $\bar{S}$. The last gap to appear in $\bar{S}$ is $33$, which commences at  $n = 153828 = (100101100011100100)_2$.

Theorem \ref{gaptheorem} is about gaps in the set $\bar{S}$ which contains integers $n$ satisfying $\Theta(n) = (0, 0, 1)$. The same method can be used for other sets defined in terms of $\Theta$ (see (\ref{sxyzdef}) below). Hajdu and Papp showed that the maximum gap length in any of these sets is $42$.
\end{remark}

\begin{figure}
\centering
\includegraphics[width=13cm]{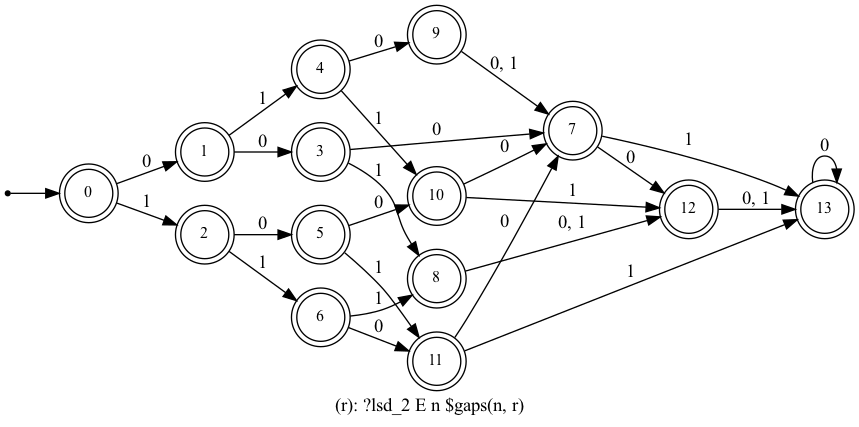}
\caption{Automaton accepting gap lengths in $\bar{S}$}
\label{sbargaps}
\end{figure}

\begin{figure}
\centering
\includegraphics[width=13cm]{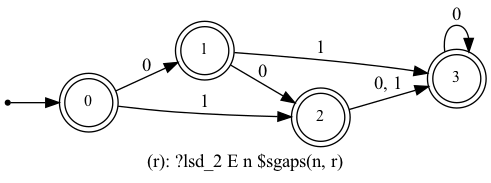}
\caption{Automaton accepting gap lengths in $S$}
\label{sgaps}
\end{figure}

\bigskip

\section{A formula for $\bar{S}(2^k)$}
\label{S2k}
In this section we establish a formula for $\bar{S}(2^k)$, The method we use is taken from the papers \cite{Du_2016}, \cite{https://doi.org/10.48550/arxiv.2006.04177}, \cite{https://doi.org/10.48550/arxiv.2203.02917}. Firstly, we use Walnut to obtain a linear representation of $\bar{S}(n)$. Walnut produces matrices $M_0$ and $M_1$ and vectors $v$ and $w$ such that, if $n$ has binary representation $(n_0, n_1, \dots , n_k )$ where $n_0$ is the least significant digit, then
\begin{equation}
\label{linrep}
\bar{S}(n) = v*M_{n_0}*M_{n_1}* \dots *M_{n_k}*w
\end{equation}

In the particular case that $n = 2^k$, we have $\bar{S}(2^k) = v*M_0 ^k * M_1 * w$. The theory of linear recurrences says that there are constants $\{ c_i \}$ such that $\bar{S}(2^k)$ can be written as 
\begin{equation}
\label{s2kequn}
\bar{S}(2^k) = \sum_i c_i \lambda_i^k
\end{equation}
where $\{\lambda_i \}$ are the roots of the minimal polynomial of $M_0$. The constants $\{c_i \}$ can be determined by calculating $\bar{S}(2^k)$ for enough values of $k$ and plugging the values into (\ref{s2kequn}) to obtain a system of linear equations that can be solved.

\bigskip

\begin{theorem}
\label{s2kthm}
For $k \geq 2$,
$$
\bar{S}(2^k) =
\begin{cases}
   2^{k-3} - 2^{(k-4)/2}, & \text{if } \, \, k \equiv \{0, 2, 22 \} \pmod{24}\\
   2^{k-3} - 2^{(k-3)/2}, & \text{if } \, \, k \equiv \{1, 3 \} \pmod{24}\\
   2^{k-3} - 2^{(k-5)/2}, & \text{if } \, \, k \equiv 11 \pmod{24}\\    
   2^{k-3} + 2^{(k-4)/2}, & \text{if } \, \, k \equiv \{14, 16, 18 \} \pmod{24}\\
   2^{k-3} + 2^{(k-3)/2}, & \text{if } \, \, k \equiv 17 \pmod{24}\\
   2^{k-3} - 3*2^{(k-5)/2}, & \text{if } \, \, k \equiv 23 \pmod{24} \\
   2^{k-3}                      & \text{otherwise }
\end{cases}
$$
\end{theorem}
\begin{proof}
The Walnut command
\begin{verbatim}
eval sumfact n "?lsd_2 (j>=1) & (j<=n) & $factauto(j)":
\end{verbatim}

produces the matrices $M_0$ and $M_1$ and the vectors $v$ and $w$ mentioned above. Since the dimensions of these objects are large, we will not display them here. The minimal polynomial of $M_0$ is
\begin{equation*}
\begin{split}
h&(x) =  \, \, x^{20} - 7*x^{19} + 20*x^{18} - 30*x^{17} + 24*x^{16} - 40*x^{14} + 64*x^{13} - 128*x^{11} \\ & + 160*x^{10} - 384*x^8 + 960*x^7 - 1280*x^6 + 896*x^5 - 256*x^4 \\
&= \, \, (x - 2) * (x - 1) * (x - i - 1) * (x + i - 1) * (x - i\sqrt{2}) * (x + i\sqrt{2}) * \\
& *(x + \sqrt{2}) * (x + (-i/2 + 1/2)*\sqrt{3} + i/2 - 1/2) * \\ 
& * (x - 1/2*\sqrt{2}*\sqrt{3} + 1/2*\sqrt{2}) * (x + 1/2*\sqrt{2}*\sqrt{3} + 1/2*\sqrt{2}) * \\
& * (x + (i/2 + 1/2)*\sqrt{3} - i/2 - 1/2) * (x + (-i/2 - 1/2)*\sqrt{3} - i/2 - 1/2) * \\
& * (x - 1/2*\sqrt{2}*\sqrt{3} - 1/2*\sqrt{2}) * (x + 1/2*\sqrt{2}*\sqrt{3} - 1/2*\sqrt{2}) * \\
& * (x + (i/2 - 1/2)*\sqrt{3} + i/2 - 1/2) * (x - \sqrt{2}) * x^4
\end{split}
\end{equation*}

All roots of $h$, apart from $2$ and $1$, are of the form $\lambda_i = \sqrt{2} * (\text{a 24'th root of unity})$.  We then have for $0 \leq s < 24$,
\begin{equation}
\label{s24ks}
\begin{split}
\bar{S}(2^{24k + s}) =& \sum_j c_j \lambda_j^{24k + s} =  c_0*2^{24k + s} + c_1 + \sum_{\lambda_j \not  1, 2 }  c_j \lambda_j^{24k + s} \\
= &c_0*2^{24k + s} + c_1 + 2^{12k} *  \sum_{\lambda_j \not = 1, 2 } c_j \lambda_j^{s}
\end{split}
 \end{equation}
The values of the $ \{c_j\}$ can be determined by calculating $\bar{S}(2^k)$ for enough $k$ using (\ref{linrep}). We create a matrix $A$ having the same dimension as the number of constants. The $j$'th row of $A$ is
$$
2^{j+2}, 1, \lambda_2^{j+2}, \lambda_3^{j+2}, \dots
$$
Then the vector of constants is given by
$$
(c_0, c_1, \dots) = A^{-1}*(\bar{S}(2^2), \bar{S}(2^3), \dots ) .
$$
The calculation gives $c_0 = 1/8$ and $c_1 = 0$.  Since we are only interested in the values for $c_0$ and $c_1$, an alternative approach is to use (\ref{s24ks}) and the calculated values of, say, $\bar{S}(2^3) = 0, \bar{S}(2^{27}) = 16773120$ and $\bar{S}(2^{51}) = 281474959933440$ to create three linear equations. The equations are:
\begin{equation*}
\begin{split}
2^3 * c_0 + c_1 + \tau =& 0 \\
2^{27} * c_0 + c_1 + 2^{12} * \tau =&16773120 \\
2^{51} * c_0 + c_1 + 2^{24} * \tau =&281474959933440
\end{split}
\end{equation*}
where $\tau =  \sum_{\lambda_j \not = 1, 2 } c_j \lambda_j^{3}$. These can be solved to give $c_0 = 1/8$, $c_1 = 0$ and $\tau = -1$.

Plugging the values for $c_0$ and $c_1$ into (\ref{s24ks}), we see that, for each $0 \leq s < 24$, the terms $\sum_{\lambda_j \not = 1, 2 } c_j \lambda_j^{s}$ satisfy 
$$
\sum_{\lambda_j \not = 1, 2 } c_j \lambda_j^{s} = \bar{S}(2^s)  - 2^{s-3}
$$

As an example of the calculation that is required for each $s$, $\bar{S}(2^3) = 0$ so $\sum_{\lambda_j \not = 1, 2 } c_j \lambda_j^{3} = -1$. Hence, if $n \equiv 3 \pmod{24}$, then $\bar{S}(2^n) = 2^{n-3} - 2^{12*\lfloor n/24 \rfloor} = 2^{n-3} - 2^{(n-3)/2}$. Here $\lfloor x \rfloor$ is the floor function. A similar calculation can be made for each $s$.
\end{proof}

\bigskip

\begin{theorem}
For $k \geq 5$,
$$
\bar{S}(3*2^k) =
\begin{cases}
   3*2^{k-3} +1, & \text{if } \, \, k \equiv \{ 7, 9, 13, 18, 19 \} \pmod{24}\\
   3*2^{k-3} + 1- 2^{(k-3)/2}, & \text{if } \, \, k \equiv \{ 1, 3, 5 \} \pmod{24}\\
   3*2^{k-3} + 1- 2^{(k-2)/2}, & \text{if } \, \, k \equiv \{ 0, 4, 10 \} \pmod{24}\\
   3*2^{k-3} + 1 - 2^{(k-4)/2}, & \text{if } \, \, k \equiv \{ 2, 6, 8, 12 \} \pmod{24}\\
   3*2^{k-3} + 1 - 2^{(k-5)/2}, & \text{if } \, \, k \equiv \{ 11, 23 \}  \pmod{24}\\
   3*2^{k-3} + 1 + 2^{(k-4)/2}, & \text{if } \, \, k \equiv 14 \pmod{24} \\
   3*2^{k-3} + 1 + 2^{(k-1)/2}, & \text{if } \, \, k \equiv 15 \pmod{24}\\
   3*2^{k-3} + 1 + 3*2^{(k-4)/2}, & \text{if } \, \, k \equiv 16 \pmod{24}\\
   3*2^{k-3} + 1 + 2^{(k-3)/2}, & \text{if } \, \, k \equiv 17 \pmod{24}\\
   3*2^{k-3} + 1 - 3*2^{(k-4)/2}, & \text{if } \, \, k \equiv \{ 20, 22 \}  \pmod{24}\\
   3*2^{k-3} + 1 - 2^{(k-1)/2}, & \text{if } \, \, k \equiv 21 \pmod{24}
\end{cases}
$$
\end{theorem}
\begin{proof}
From (\ref{linrep}),
$$
\bar{S}(3*2^k) = v*M_0^k*M_1*M_1 * w
$$
We can therefore use the same approach as in theorem \ref{s2kthm}. We have,
\begin{equation}
\label{s32kequn}
\bar{S}(3*2^k) = \sum_i c_i \lambda_i^k
\end{equation}
where $\{ c_j \}$ are constants and $\{ \lambda_j \}$ are the roots of the minimal polynomial of $M_0$, which we have already calculated. The constants $c_0$ and $c_1$ can be calculated using the values of $\bar{S}(3*2^3) = 3, \bar{S}(3*2^{27}) = 50327553$ and $\bar{S}(3*2^{51}) = 844424913354753$ to create three linear equations. The equations are:
\begin{equation*}
\begin{split}
2^3 * c_0 + c_1 + \tau =& 3 \\
2^{27} * c_0 + c_1 + 2^{12} * \tau =&50327553 \\
2^{51} * c_0 + c_1 + 2^{24} * \tau =& 844424913354753
\end{split}
\end{equation*}
where $\tau =  \sum_{\lambda_j \not = 1, 2 } c_j \lambda_j^{3}$. These can be solved to give $c_0 = 3/8$, $c_1 = 0$ and $\tau = -1$. As before, we have for $0 \leq s < 24$,
\begin{equation}
\label{xxx}
\begin{split}
\bar{S}(3*2^{24k + s}) =& \sum_j c_j \lambda_j^{24k + s} =  3/8*2^{24k + s} + 1 + \sum_{\lambda_j \not  1, 2 }  c_j \lambda_j^{24k + s} \\
= &3*2^{24k + s - 3} + 1 + 2^{12k} *  \sum_{\lambda_j \not = 1, 2 } c_j \lambda_j^{s}
\end{split}
\end{equation}
The formula for $\bar{S}(3*2^k)$ can be deduced once (\ref{s32kequn}) is used to calculate $\sum_{\lambda_j \not = 1, 2 } c_j \lambda_j^{s}$ for each $0 \leq s < 24$.
\end{proof}

\bigskip

\section{A bound for $\bar{S}(n)$}
\label{bound}

In this section we establish a best possible bound on $\bar{S}(n)$. We borrow an argument from the paper by Hajdu and Papp \cite{Hajda2018}. 

We introduce the following sets for any triple $(x, y, z)$ of binary digits:
\begin{equation}
\label{sxyzdef}
\begin{split}
\bar{S}_{(x,y,z)} :=& \{ n: \Theta(n) = (x, y, z) \}, \\
\bar{S}_{(x,y,z)}(n) :=& \{ r \leq n: \Theta(r) = (x, y, z) \}, \\
\bar{S}_{(x,y,z)}(m, n) :=& \{ m < r \leq n: \Theta(r) = (x, y, z) \}
\end{split}
\end{equation}
Then $\bar{S}_{(0, 0, 1)} = \bar{S}$ and $\bar{S}_{(0, 0, 1)}(n) = \bar{S}(n)$. 
\bigskip

Note that our set $\bar{S}_{(x,y,z) }$ is equivalent to the set  $H^{(\alpha, \beta)}$ used by Hajdu and Papp, where $\alpha = x$ and $\beta = 3^y(-1)^z$.

We note that the results from the previous section show that 
$$
\bar{S}(2^k) = 1/8*2^{k} +  \mathcal{O}(\sqrt{2^k}) \text{ \, \, and \, \, } \bar{S}(3*2^k) = 1/8(3*2^{k}) +  \mathcal{O}(\sqrt{2^k}) 
$$
The same bound can be shown to apply for other sets of numbers defined by $(\gamma, \alpha_3, \alpha_5 )$. In summary, as for the case when $(x, y, z ) = (0, 0, 1)$,  the set $\bar{S}_{(x,y,z)} (n)$ is accepted by an automaton constructed from the automata "gamma", "a3" and "a5" displayed earlier. For example, the set of $\{n: \Theta(n) = (1, 1, 1) \}$ is accepted by the automaton defined by the Walnut command:
\begin{verbatim}
def factauto111 "?lsd_2 ~$gamma(n) & ~$a3(n) & ~$a5(n)":
\end{verbatim}

As before, a linear representation of this automaton in terms of the matrices $M_0$ and $M_1$ and the vectors $v$ and $w$ comes from the command:
\begin{verbatim}
eval sumfact111 n "?lsd_2 (j>=1) & (j<=n) & $factauto111(j)":
\end{verbatim}

For all choices of $(x, y, z)$, the matrix $M_0$ that Walnut produces is the same as the matrix $M_0$ that was produced by the automaton "sumfact" that was used in analysing the $\Theta(n) = (0, 0, 1)$ case. All the previous analysis related to the case $\Theta(n) = (0, 0, 1)$ therefore applies to the other choices of $(x, y, z)$.
As a result, the same methods show that, for any $(x, y, z)$, 
\begin{equation}
\label{sbarxyz}
\bar{S}_{(x,y,z)}(2^k) = 1/8*2^{k} +  \mathcal{O}(\sqrt{2^k}) \text{ \, \, and \, \, } \bar{S}_{(x,y,z)}(3*2^k) = 1/8(3*2^{k}) +  \mathcal{O}(\sqrt{2^k}) .
\end{equation}
\bigskip

It then follows that, for any choice of $(x, y, z )$ and $k$,
\begin{equation}
\label{sbarsubreg}
\begin{split}
\bar{S}_{(x,y,z)}(2^k, 2^{k+1}) &= 1/8*2^k + \mathcal{O}(\sqrt{2^{k}}) \\
\bar{S}_{(x,y,z)}(2^{k+1}, 3*2^{k}) &= 1/8*2^k + \mathcal{O}(\sqrt{2^{k}}) \\
\bar{S}_{(x,y,z)}(3*2^{k-2}, 2^{k}) &= 1/8*2^k + \mathcal{O}(\sqrt{2^{k}}) .
\end{split}
\end{equation}

\bigskip

\begin{lemma}[Hajdu and Papp Lemma 3.1]
\label{hp3.1}
For any  integer $k$ with $k \geq 1$, we have $\gamma(t*2^k + i) = \gamma( i ) + \gamma( t*2^k ) \pmod{2}$ for $0 \leq i < 2^k$.
\end{lemma}

\bigskip

For the following lemma, we divide the region $[0 , 2^s )$ into four equal sized subregions given by:
\begin{equation}
\label{subreg}
\begin{split}
I_1 :=& \{ i : 0 \leq i < 2^{s-2} \} \\
I_2 :=& \{ i: 2^{s-2} \leq i < 2^{s-1} \} \\
I_3 :=& \{ i: 2^{s-1} \leq i < 3*2^{s-2} \} \\
I_4 :=& \{ i: 3*2^{s-2} \leq i < 2^s \} .
\end{split}
\end{equation}

\bigskip
\begin{lemma}[Hajdu and Papp Lemma 3.3]
\label{hp3.3}
Let $k$ be an integer and $t$ an odd integer with $t \geq 1$. Then, for each $j \in \{1, 2, 3, 4 \}$,  there exist numbers $c(t, j) \in \{ 1, 3, 5, 7 \}$ such that
$$
3^{\alpha_3(t*2^k + i)} * (-1)^{\alpha_5(t*2^k + i)} = c(t, j) * 3^{\alpha_3 (i) } * (-1)^{\alpha_5 (i)} \pmod 8
$$
for all $i \in I_j$.
\end{lemma}

\bigskip

\begin{lemma}
\label{orderlem}
For integers $r, s$ and $t$ satisfying  $0 \leq r < s$, $0 \leq t < 2^s$, $2 \not | t $, 
$$
\bar{S}(t*2^{s}, \, \, t*2^{s} + 2^{r}) = 1/8 * 2^{r} +  \mathcal{O}(\sqrt{2^{r}})
$$
\end{lemma}
\begin{proof}
\begin{equation*}
\begin{split}
\bar{S}(t*2^{s}, \, \, t*2^{s} + 2^{r}) &= \# \{ i : 0 \leq i < 2^r, \Theta (t*2^{s} + i ) = (0, 0, 1) \} \\
& = \sum_{j = 1}^{4} \{ i \in I_j : \Theta (t*2^{s} + i ) = (0, 0, 1) \} 
\end{split}
\end{equation*}
where $I_j$ are the regions defined in ( \ref{subreg} ). From lemma \ref{hp3.1}, if $0 \leq i < 2^s$, 
$$
\Theta_1 (t*2^{s} + i ) = 0 \iff \Theta_1 ( i ) = - \Theta_1 (t*2^s)
$$
From lemma \ref{hp3.3}, if $i \in I_j$, 
\begin{equation*}
\begin{split}
(\Theta_2 (t*2^{s} + i ), \, \, \,& \Theta_3 (t*2^{s} + i) )  = (0, 1) \\
& \iff c(t, j) * 3^{\alpha_3 (i) } * (-1)^{\alpha_5 (i)} \equiv - 1 \pmod 8 \\
& \iff (\Theta_2 (i) , \Theta_3 (i) ) = ( d_2 (t, j), d_3 (t, j))
\end{split}
\end{equation*}
for some $d_1, d_2 \in \{0, 1 \}$. Hence,
\begin{equation*}
\begin{split}
\bar{S}(t*2^{s}, \, \, t*2^{s} + 2^{r}) = &\sum_{j=1}^{4} \# \{ i \in I_j : \Theta (i) = (- \Theta_1 (t*2^k), d_2 (t, j), d_3 (t, j)) \} \\
= & \, \, \bar{S}_{(x, y, z)_1} (2^{r-2}) + \bar{S}_{(x, y, z)_2 }(2^{r-2}, 2^{r-1}) \, \, + \\
& + \bar{S}_{( x, y, z )_3} (2^{r-1}, 3*2^{r-2}) + \bar{S}_{(x, y, z)_4} (3*2^{r-2}, 2^r) \\
= & \, \,4*(1/8 * 2^{r-2} +  \mathcal{O}(\sqrt{2^{r}}) ) \\
= & \, \, 1/8*2^r +  \mathcal{O}(\sqrt{2^{r}}) .
\end{split}
\end{equation*}
where $(x, y, z)_j = (- \Theta_1 (t*2^k), d_2 (t, j), d_3 (t, j))$ for $j \in \{1, 2, 3, 4 \}$ and we have used (\ref{sbarsubreg}).
\end{proof}

\bigskip

The argument in the following theorem comes from theorem 2.3 in the paper by Hajdu and Papp.

\bigskip

\begin{theorem}
$\bar{S} (n) = 1/8 * n + \mathcal{O}( \sqrt{n})$
\end{theorem}
\begin{proof}
We first note that (\ref{sbarxyz}) implies that for any $(x, y, z)$,
\begin{equation}
\label{sbarxyxint}
\begin{split}
\bar{S}_{(x,y,z)} (2^{k-2}, 2^{k-1}) =& 1/8*2^{k-2} +  \mathcal{O}(\sqrt{2^k}) \\
\bar{S}_{(x,y,z)} (2^{k -1}, 3*2^{k - 2}) =& 1/8*2^{k -2} +  \mathcal{O}(\sqrt{2^k}) \\
\bar{S}_{(x,y,z)} (3*2^{k -2}, 2^{k}) =& 1/8*2^{k - 2} +  \mathcal{O}(\sqrt{2^k}) .
\end{split}
\end{equation}

Let $n$ be an integer with binary representation $n = \sum_1^j 2^{f_i}$, where $f_1 >  \dots > f_j \geq 0$. Then,
$$
\bar{S}(n) = \bar{S}(2^{f_1}) + \bar{S}(2^{f_1}, 2^{f_1} + 2^{f_2}) + \dots + \bar{S}(2^{f_1} + \dots + 2^{f_{j-1}}, 2^{f_1} + \dots + 2^{f_j})
$$
Any of the terms above can be written as
$$
\bar{S}(2^{f_1} + \dots + 2^{f_{l-1}}, \, \, 2^{f_1} + \dots + 2^{f_j}) = \bar{S}(t*2^{f_l}, \, \, t*2^{f_l} + 2^{f_{l+1}})
$$
where $t < 2^{f_l}$ is odd. Since,
$$
 \bar{S}(t*2^{f_l}, \, \, t*2^{f_l} + 2^{f_{l+1}}) = 1/8 * 2^{f_{l+1}} +  \mathcal{O}(\sqrt{2}^{f_l + 1})
$$
by lemma \ref{orderlem}, we have
\begin{equation*}
\begin{split}
\bar{S}(n) =& \sum_{l=1}^j 1/8 * 2^{f_l} +  \mathcal{O}(\sqrt{2}^{f_l} ) \\
                =& 1/8 * n + \mathcal{O} (( \sqrt2^{f_1 + 1} - 1)/ (\sqrt{2} - 1) ) \\
              = & 1/8 * n +  \mathcal{O} ( \sqrt{2}^{f_1} ) .
\end{split}
\end{equation*}
Since $f_1 \leq \log_2 (n)$ we have
$$
\bar{S} (n) = 1/8 * n + \mathcal{O}(\sqrt{n}) .
$$
\end{proof}
\bigskip
\bibliographystyle{plain}
\begin{small}
\bibliography{Factorial}
\end{small}

\end{document}